\documentclass[12pt]{article}
\usepackage{amsthm,amsmath,amssymb,latexsym,amscd}
\newcommand{\g}{\frak{g}}
\newcommand{\h}{\frak{h}}

\newcommand{\A}{\mathcal{A}}
\newcommand{\F}{\mathcal{F}}

\newcommand{\LM}{\mathcal{L}\mathcal{M}}

\renewcommand{\P}{\mathcal{P}}

\newcommand{\Hom}{\mathrm{Hom}}

\newcommand{\lbar}{\overline}

\newcommand{\Ass}{\mathcal{A}ss}
\newcommand{\Com}{\mathcal{C}om}

\newcommand{\Dias}{\mathcal{D}ias}
\newcommand{\Lie}{\mathcal{L}ie}
\newcommand{\Lod}{\mathcal{L}od}

\newcommand{\Perm}{\mathcal{P}erm}
\newcommand{\Poiss}{\mathcal{P}oiss}
\newcommand{\Poly}{\mathcal{A}^{poly}}

\newcommand{\p}{\prime}
\newcommand{\pa}{\partial}

\newcommand{\ot}{\otimes}

\newcommand{\vd}{\vdash}
\newcommand{\dv}{\dashv}

\newtheorem{definition}{Definition}[section]
\newtheorem{lemma}[definition]{Lemma}

\newtheorem{proposition}[definition]{Proposition}
\newtheorem{theorem}[definition]{Theorem}
\newtheorem{corollary}[definition]{Corollary}
\newtheorem{remark}[definition]{Remark}
\newtheorem{example}[definition]{Example}
\date{}

\begin{document}

\title{Noncommutative Poisson brackets on Loday algebras
and related deformation quantization}
\author{K. UCHINO\footnote{Freelance; Tokyo Japan.}}
\maketitle
\abstract
{
Given a Lie algebra, there uniquely exists a Poisson
algebra which is called a Lie-Poisson algebra
over the Lie algebra.
We will prove that given a Loday/Leibniz algebra
there exists uniquely a noncommutative Poisson algebra
over the Loday algebra.
The noncommutative Poisson algebras are called
the Loday-Poisson algebras.
In the super/graded cases,
the Loday-Poisson bracket
is regarded as a noncommutative version
of classical (linear) Schouten-Nijenhuis bracket.
It will be shown that the Loday-Poisson algebras
form a special subclass of Aguiar's dual-prePoisson algebras.
We also study a problem of deformation quantization
over the Loday-Poisson algebra.
It will be shown that
the polynomial Loday-Poisson algebra
is deformation quantizable and
that the associated quantum algebra
is Loday's associative dialgebra.
}
%%%%%%%%%%%%%%%%%%%%%%
%%%%%%%%%%%%%%%%%%%%%%
\section{Introduction}
%%%%%%%%%%%%%%%%%%%%%%
%%%%%%%%%%%%%%%%%%%%%%

A Loday algebra (also called Leibniz algebra \cite{Lod1,Lod2})
is an algebra equipped with a binary bracket
satisfying the Leibniz identity,
$$
[x_{1},[x_{2},x_{3}]]=[[x_{1},x_{2}],x_{3}]+[x_{2},[x_{1},x_{3}]].
$$
This new type algebra was introduced by J-L. Loday.
Hence it is called a Loday algebra.
A Lie algebra is clearly a Loday algebra whose bracket
is anti-commutative.
However the Loday bracket is not anti-commutative in general.\\
\indent
The Loday algebras arise
in several area of mathematics,
for example, as Loday algebroids
in Nambu-mechanics (\cite{ILMP});
as Courant algebroids in Poisson geometry;
in representation theory of Lie algebras
and so on.
\medskip\\
\indent
Because Loday algebras are considered to be
noncommutative analogues of Lie algebras,
it is natural to ask what
type of algebra is noncommutative analogue
of Lie-Poisson algebra.
The aim of this note is to solve this question.
\medskip\\
\indent
In first, we recall classical Lie-Poisson algebras.
If $\g$ is a Lie algebra, then
the polynomial functions $S(\g)$ becomes a Poisson algebra,
that is, the Lie-Poisson algebra.
The Lie-Poisson bracket
is characterized by the following 4-conditions;
(C1)
The Lie-Poisson algebras are Lie- and Commutative-algebras.
(C2)
The Lie-Poisson bracket is a biderivation
with respect to the associative multiplication on $S(\g)$.
(C3)
The base algebra $\g$ is a Lie subalgebra of $S(\g)$.
(C4)
The Lie-Poisson bracket on $S(\g)$ is
{\em free} over the Lie bracket on $\g$.\\
\indent
Let us consider noncommutative analogues of the conditions (C1)-(C4).
We assume that $\g$ is a Loday algebra
and that $\A$ is a certain algebra generated from $\g$.
The noncommutative analogues of (C1)-(C4) are respectively
as follows:
\begin{description}
\item[(A1)] The algebra $\A$ is a Loday- and associative-algebra,
not necessarily commutative.
\item[(A2)] The Loday bracket on $\A$ is a biderivation
with respect to the associative multiplication.
\item[(A3)] The base algebra $\g$ is a Loday subalgebra of $\A$.
\item[(A4)] The Loday bracket on $\A$ is {\em free}
over the bracket on $\g$.
\end{description}
We notice that
Aguiar's dual-prePoisson algebras (\cite{A}) satisfy (A1)-(A3).
We will prove that given a Loday algebra
there exists a dual-prePoisson algebra
satisfying also (A4) (\textbf{Theorem} \ref{theorema} below).
We call this special dual-prePoisson algebra
a \textbf{Loday-Poisson algebra},
which is considered to be a noncommutative
version of Lie-Poisson algebra.
As an application, we will study {\em super} Loday-Poisson brackets,
which are considered as noncommutative analogues
of classical (linear) Schouten-Nijenhuis brackets
(see Section 4.2 below).\\
\indent
To understand the dual-prePoisson algebras
we will study a category of linear mappings $\LM$ in \cite{LP}.
It is known that Loday algebras are {\em Lie gebras} in $\LM$
(so-called Lie objects).
We will introduce the notion of {\em Poisson gebra} in $\LM$,
which are called Poisson objects (see Section 3 below).
It will be shown that the Loday-Poisson algebra is
the {\em free} Poisson gebra defined over a Lie gebra.\\
\indent
We also study an algebraic quantization
of polynomial Loday-Poisson algebras.
It will be shown that
the polynomial Loday-Poisson algebras
are deformation quantizable in the sense of Kontsevich \cite{Kont}
and that the quantum algebras
are \textbf{associative dialgebras} in \cite{Lod2}.
Namely, given a Loday-Poisson algebra of polynomials,
there exists a unique associative dialgebra parameterized by $\hbar$
whose classical limit is the Loday-Poisson algebra.
It is a difficult problem whether Nambu-Poisson brackets are quantizationable.
Since the Nambu-Poisson brackets are equivalently Loday brackets,
our results provide a hope that Nambu-Poisson brackets are quantizationable.
\medskip\\
\indent
The paper is organized as follows.
In Section 2, we recall three noncommutative algebras,
i.e., Loday algebras,
perm-algebras of Chapoton \cite{Chap,Chap2}
and dual-prePoisson algebras.
In Section 3, we study a category of linear mappings.
In Section 4, we give the main theorem of this note,
namely, the problem above is solved.
In Section 5, we discuss the deformation quantization
for the Loday-Poisson algebras.
\medskip\\
\noindent
\textbf{Acknowledgement}.
The author would like to thank very
much referees for useful comments
and also thank Professor Alan Weinstein
for kind support.
%%%%%%%%%%%%%%%%%%%%%%%%%%%%%
%%%%%%%%%%%%%%%%%%%%%%%%%%%%%
\section{New type algebras}
%%%%%%%%%%%%%%%%%%%%%%%%%%%%%
%%%%%%%%%%%%%%%%%%%%%%%%%%%%%
\subsection{Loday/Leibniz algebras (\cite{Lod1, Lod2})}

A Loday algebra $(\g,[,])$ is, by definition,
a vector space equipped with a binary bracket product
which satisfies the Leibniz identity,
$$
[x,[y,z]]=[[x,y],z]+[y,[x,z]],
$$
where $x,y,z\in\g$.
If the bracket is anti-commutative, then
the Loday algebra is a Lie algebra.
Thus, Loday algebras are considered to be noncommutative
analogues of Lie algebras.\\
\indent
We consider a subspace of the Loday algebra
composed of the symmetric brackets,
$$
\g^{ann}:=\{[x,y]+[y,x] \ | \ x,y\in\g\}.
$$
This becomes a two side ideal of the Loday algebra
because
\begin{eqnarray*}
\ [[x,x],y]&=&[x,[x,y]]-[x,[x,y]]=0,\\
\ [x,[y,y]]&=&[[x,y],y]+[y,[x,y]],
\end{eqnarray*}
for any $x,y\in\g$.
Hence the quotient space $\g_{Lie}:=\g/\g^{ann}$ becomes a Lie algebra.
The projection $\g\to\g_{Lie}$, which is called a \textbf{Liezation},
is a universal arrow, that is,
$$
\Hom_{Lod}(\g,\Box{\h})\cong\Hom_{Lie}(\g_{Lie},\h),
$$
where $\h$ is an arbitrary Lie algebra
and $\Box$ is the forgetful functor
from the category of Lie algebras
to the one of Loday algebras.
One can define a Lie algebra action $\g_{Lie}\ot\g\to\g$ by
\begin{equation}\label{fact}
ad(\bar{x})(y):=[x,y].
\end{equation}
The Loday bracket is identified with this action.
\medskip\\
\indent
The free Loday algebra over a vector space
has the following form,
$$
\F_{\Lod}(V)=\bar{T}V:=\bigoplus_{n\ge 1}V^{\ot n}.
$$
Here the Loday bracket is defined by
$$
[x_{1},[x_{2},[...,[x_{n-1},x_{n}]]]]:=x_{1}\ot\cdot\cdot\cdot\ot x_{n},
$$
for any $x_{\cdot}\in V$.
%%%%%%%%%%%%%%%%%%%%%%%%%%%%%%%%%%%%%%%%%%%%%%
\subsection{Perm-algebras (\cite{Chap,Chap2})}
%%%%%%%%%%%%%%%%%%%%%%%%%%%%%%%%%%%%%%%%%%%%%%
An associative algebra $(A,*)$ is called a permutation algebra,
or perm-algebra for short, if $A$ satisfies
\begin{equation}\label{defperm}
(x*y)*z=(y*x)*z,
\end{equation}
where $x,y,z\in A$.\\
\indent
We consider a monomial on the perm-algebra.
$x_{1}*\cdots x_{n-1}*x_{n}\in A$.
Up to the $(-)*x_{n}$, $x_{1}*\cdots*x_{n-1}$
is regarded as a monomial of commutative algebra.
This implies that the free perm-algebra
over a vector space $V$ has the
following form,
$$
\F_{\Perm}(V)=S(V)\ot V,
$$
where $S(V)$ is the free ``unital" commutative algebra over $V$,
or the polynomial algebra over $V$.
The perm-multiplication on $\F_{\Perm}(V)$
is defined by
\begin{equation}\label{fp}
(f\ot x)*(g\ot y):=fxg\ot y,
\end{equation}
where $f,g\in S(V)$ and $x,y\in V$.
\begin{example}
Let $(C,d)$ be a differential graded (dg) commutative algebra.
Define a new product on $C$ by
$$
x*y:=-(-1)^{|x|}(dx)y.
$$
Then $(C,*)$ becomes a graded perm-algebra.
\end{example}
%%%%%%%%%%%%%%%%%%%%%%%%%%%%%%%%%%%%%
\subsection{Dual-prePoisson/Aguiar algebras (\cite{A})}
%%%%%%%%%%%%%%%%%%%%%%%%%%%%%%%%%%%%%
Let $(A,*,\{,\})$ be a perm-algebra
equipped with a Loday bracket.
It is called a dual-prePoisson algebra\footnote{
The operad of dual-prePoisson algebras is the Koszul ``dual" of
the one of pre-Poisson algebras.
As such the name is.
}, if the following three conditions are satisfied,
\begin{eqnarray}
\label{dpp1} \{x,y*z\}&=&\{x,y\}*z+y*\{x,z\},\\
\label{dpp2} \{x*y,z\}&=&x*\{y,z\}+y*\{x,z\},\\
\label{dpp3} \{x,y\}*z&=&-\{y,x\}*z,
\end{eqnarray}
where $x,y,z\in A$.
We call the axioms (\ref{dpp1}) and (\ref{dpp2})
the \textbf{biderivation} conditions of dual-prePoisson algebra.
In the graded cases,
(\ref{dpp1})-(\ref{dpp3}) have the following form,
\begin{eqnarray*}
\{x,y*z\}&=&\{x,y\}*z+(-1)^{(|x|-n)|y|}y*\{x,z\},\\
\{x*y,z\}&=&x*\{y,z\}+(-1)^{|x||y|}y*\{x,z\},\\
\{x,y\}*z&=&-(-1)^{(|x|-n)(|y|-n)}\{y,x\}*z,
\end{eqnarray*}
where $|x|$, $|y|$ are the degrees of elements
and $n$ is the one of the bracket.
The perm-multiplication also satisfies the natural sign convention,
$$
x*y*z=(-1)^{|x||y|}y*x*z,
$$
where we put $|*|:=0$.\\
\indent
We sometimes call the dual-prePoisson algebra
an \textbf{Aguiar algebra}, because the name ``dual-prePoisson" is long. 
\medskip\\
\indent
We consider a subspace of the dual-prePoisson algebra
generated by the symmetric elements,
$$
A^{ann}:=<x*y-y*x, \ \  \{x,y\}+\{y,x\}>.
$$
This is a two side ideal of the dual-prePoisson
algebra. Hence the quotient space $A/A^{ann}$
becomes a Poisson algebra.
One can easily check that $A\to A/A^{ann}$ is
the universal arrow, like the Liezation $\g\to\g_{Lie}$.
So we call the projection $A\to A/A^{ann}$ a \textbf{Poissonization}
of dual-prePoisson algebra.\\
\indent
From the axioms of dual-prePoisson algebras,
we obtain $A^{ann}*A=0$ and $\{A^{ann},A\}=0$.
Under a mild assumption, $A$ is
considered to be a semi-direct product algebra
$$
A=A_{Poiss}\ltimes A^{ann},
$$
where $A_{Poiss}$ is the result of the Poissonization $A\to A_{Poiss}$.
Conversely, given a Poisson algebra $P$
and a left Poisson module $M$,
the semi-direct product $P\ltimes M$
becomes a dual-prePoisson algebra,
whose dual-prePoisson products are defined by
\begin{eqnarray}
\label{mdp1}(p_{1}+m_{1})*(p_{2}+m_{2})&:=&p_{1}p_{2}+p_{1}\cdot m_{2},\\
\label{mdp2}\{p_{1}+m_{1},p_{2}+m_{2}\}&:=&\{p_{1},p_{2}\}+\{p_{1},m_{2}\},
\end{eqnarray}
where $p_{1},p_{2}\in P$, $m_{1},m_{2}\in M$
and where $\{p_{1},p_{2}\}$ is the Poisson bracket on $P$
and $p_{1}\cdot m_{2}$, $\{p_{1},m_{2}\}$ are
the left-module structures.
\begin{example}
Let $(P,d,\{,\}^{\p})$ be a dg Poisson algebra
with a Poisson bracket $\{,\}^{\p}$.
Define the new products by
\begin{eqnarray*}
x*y&:=&-(-1)^{|x|}(dx)y,\\
\{x,y\}&:=&-(-1)^{|x|}\{dx,y\}^{\p}.
\end{eqnarray*}
Then $(P,*,\{,\})$ becomes a graded dual-prePoisson algebra.
\end{example}
\indent
We consider the free dual-prePoisson algebra.
By the biderivation properties (\ref{dpp1}) and (\ref{dpp2}),
a monomial of dual-prePoisson algebra, $m$,
is decomposed into the form,
\begin{equation}\label{mmn}
m=\sum_{m^{\p},n}m^{\p}*n,
\end{equation}
where $m^{\p}$ and $n$ are monomials
and where $n$ has no $*$.
For example,
\begin{multline*}
\{\{x,y*z\},w\}=\\
=\{x,y\}*\{z,w\}+z*\{\{x,y\},w\}+
y*\{\{x,z\},w\}+\{x,z\}*\{y,w\},
\end{multline*}
where $\{z,w\}$, $\{\{x,y\},w\}$, $\{\{x,z\},w\}$
and $\{y,w\}$ are $n$ in (\ref{mmn}).\\
\indent
Since a dual-prePoisson algebra is a Loday algebra,
$n$ is regarded as a monomial of Loday algebra.
From the axioms of dual-prePoisson algebra,
$(x*y)*z=(y*x)*z$ and $\{x,y\}*z=-\{y,x\}*z$,
$m^{\p}$ is regarded as a monomial of Poisson algebra
up to $(-)*A$.
From this observation, we obtain
\begin{proposition}\label{prop0}
Let $V$ be a vector space.
The free dual-prePoisson
algebra over $V$ has the following form,
\begin{equation}\label{fdpp}
\F_{\Poiss}(V)\ot\F_{\Lod}(V),
\end{equation}
where $\F_{\Poiss}(V)$ is the free ``unital" Poisson algebra.
\end{proposition}
\begin{proof}
We will give a proof of the proposition in Section 4.
\end{proof}
\begin{proposition}
$\lbar{\F}_{\Poiss}(V)\cong
\Big(\F_{\Poiss}(V)\ot\F_{\Lod}(V)\Big)_{Poiss}$,
where $\lbar{\F}_{\Poiss}(V)$ is the free Poisson algebra
of nonunital.
\end{proposition}
\begin{proof}
Because the Poissonization is universal.
\end{proof}

%%%%%%%%%%%%%%%%%%%%%%%%%%%%%%%%%%%%%
\section{Category of linear mappings}
%%%%%%%%%%%%%%%%%%%%%%%%%%%%%%%%%%%%%

The category of linear maps (\cite{LP}),
which is denoted by $\LM$, is a category
whose objects are linear maps $\rho:V_{1}\to V_{0}$
and morphisms $(F_{1},F_{0})$ are commutative diagrams
of liner maps,
$$
\begin{CD}
V_{1}@>{F_{1}}>>V^{\p}_{1} \\
@V{\rho}VV @VV{\rho^{\p}}V \\
V_{0}@>{F_{0}}>>V^{\p}_{0}.
\end{CD}
$$
One can define a tensor product on $\LM$ by
\begin{eqnarray}
\label{defotlm1}&&\rho\ot_{\LM}\rho^{\p}:=\rho\ot 1+1\ot \rho^{\p},\\
\label{defotlm2}&&\rho\ot_{\LM}\rho^{\p}:(V_{1}\ot V^{\p}_{0})
\oplus(V_{0}\ot V^{\p}_{1})\to V_{0}\ot V^{\p}_{0}.
\end{eqnarray}
If $F:=(F_{1},F_{0})$ and $G:=(G_{1},G_{0})$ are morphisms in $\LM$,
then the tensor product of $F$ and $G$ has the following form,
\begin{equation}\label{defotlm3}
F\ot_{\LM}G=(F_{1}\ot G_{0}\oplus F_{0}\ot G_{1},F_{0}\ot G_{0}).
\end{equation}
It is known that $(\LM,\ot_{\LM})$ becomes a symmetric monoidal category.
\medskip\\
\indent
A Lie algebra object (shortly, Lie object)
is by definition a Lie (al)gebra in $\LM$.
That is, a Lie object is a linear mapping
$\rho:\g_{1}{\to}\g_{0}$
equipped with a bracket morphism $\mu=(\mu_{1},\mu_{0})$,
$$
\begin{CD}
\g_{1}\ot\g_{0}\oplus\g_{0}\ot\g_{1}@>{\mu_{1}}>>\g_{1}\\
@VVV @VVV\\
\g_{0}\ot\g_{0}@>{\mu_{0}}>>\g_{0},
\end{CD}
$$
satisfying the axioms of Lie algebras, i.e.,
skewsymmetry and Jacobi law,
\begin{eqnarray*}
&&\mu(21)=-\mu,\\
&&\mu(1\ot_{\LM}\mu)=\mu(\mu\ot_{\LM}1)+\mu(1\ot_{\LM}\mu)(213),
\end{eqnarray*}
where $(21),(213)$ are usual actions of
symmetric groups $S_{2},S_{3}$, respectively.
The morphisms between Lie objects are defined
by the usual manner.\\
\indent
If $\rho:\g_{1}\to\g_{0}$ is a Lie object equipped with
a bracket $(\mu_{1},\mu_{0})$,
then $(\g_{0},\mu_{0})$ becomes a Lie algebra
and $(\g_{1},\mu_{1})$ becomes a $\g_{0}$-bimodule satisfying
the equivariant condition,
$$
\rho[\rho(x),y]_{\mu_{1}}=[\rho(x),\rho(y)]_{\mu_{0}},
$$
for any $x,y\in\g_{1}$.
The natural bracket, $[x,y]_{Lod}:=[\rho(x),y]_{\mu_{1}}$,
becomes a Loday bracket on $\g_{1}$.
\begin{proposition}
[\cite{LP}]
The Liezation $\g\to\g_{Lie}$ is a Lie object
and an arrow $\g\Rightarrow(\g\to\g_{Lie})$
is an adjoint functor with respect to
the forgetful functor, $(\h_{1}\to\h_{0})\Rightarrow h_{1}$, i.e.,
$$
\Hom_{Lieobj}(\g\to\g_{Lie},\h_{1}\to\h_{0})\cong\Hom_{Lod}(\g,\h_{1}),
$$
where LHS is the space of morphisms between Lie objects.
\label{proplp}
\end{proposition}
\begin{proof}
(Sketch)
Let $\g$ be a Loday algebra.
The Liezation $\g\to\g_{Lie}$ has a bracket
\begin{eqnarray*}
\mu=
\left(
\begin{array}{c}
\mu_{1}\\
\mu_{0}
\end{array}
\right),
\end{eqnarray*}
where $\mu_{0}$ is the Lie bracket on $\g_{Lie}$
and $\mu_{1}$ is the action of $\g_{Lie}$ to $\g$.
From (\ref{defotlm3}),
\begin{eqnarray*}
\mu(1\ot_{\LM}\mu)=\mu
\left(
\begin{array}{c}
1\ot\mu_{0}\oplus 1\ot\mu_{1}\\
1\ot\mu_{0}
\end{array}
\right)
=
\left(
\begin{array}{c}
\mu_{1}(1\ot\mu_{0}\oplus 1\ot\mu_{1})\\
\mu_{0}(1\ot\mu_{0})
\end{array}
\right).
\end{eqnarray*}
By a direct computation, one can check the Jacobi identity.\\
\indent
Given a Loday algebra homomorphism $\g\to\h_{1}$
in $\Hom_{Lod}(\g,\h_{1})$,
by the universality of the Liezation $\g\to\g_{Lie}$,
there exists a unique Lie algebra morphism
of $\g_{Lie}\to\h_{0}$.
Thus we obtain a morphism of Lie objects.
\end{proof}

We introduce the notion of {\em Poisson object}.
The Poisson objects are by definition
linear mappings $P_{1}\to P_{0}$ equipped with
the pairs of commutative associative
multiplications $\nu$ and Lie brackets $\mu$ satisfying
the distributive law,
$$
\mu(\nu\ot_{\LM}1)=\nu(1\ot_{\LM}\mu)+\nu(1\ot_{\LM}\mu)(213).
$$
The morphisms between Poisson objects
are also defined by the usual manner.\\
\indent
If $P_{1}\to P_{0}$ is a Poisson object,
then $P_{0}$ becomes a Poisson algebra
and $P_{1}$ becomes a $P_{0}$-bimodule satisfying
certain equivariant conditions.
Then, $P_{1}$ becomes a dual-prePoisson algebra
by a similar method with (\ref{mdp1}) and (\ref{mdp2}).
The Poisson version of  Proposition \ref{proplp} is as follows.
\begin{proposition}
If $A$ is a dual-prePoisson algebra,
the Poissonization $A\to A_{Poiss}$
is a Poisson object in $\LM$,
and in a similar way the arrow
$A\Rightarrow(A\to A_{Poiss})$ has the universality
$$
\Hom_{Poissobj}(A\to A_{Poiss},P_{1}\to P_{0})
\cong
\Hom_{Agu}(A,P_{1}),
$$
where LHS (resp. RHS) is the space of morphisms
between Poisson objects
(resp. Aguiar/dual-prePoisson algebras)
and where $P_{1}\to P_{0}$ is an arbitrary Poisson object.
\end{proposition}

%%%%%%%%%%%%%%%%%%%%%%%%%%%%%%%%
\section{Loday-Poisson algebras}
%%%%%%%%%%%%%%%%%%%%%%%%%%%%%%%%
In this section we solve the question in Introduction.

\subsection{Polynomial cases}

Let $\g$ be a Loday algebra.
We define a space of nonstandard polynomial functions
on the dual space $\g^{*}$ as
$$
\Poly(\g^{*}):=S(\g_{Lie})\ot \g.
$$
Here $1\ot\g(\cong\g)$ is the space of linear functions on $\g^{*}$.
An associative multiplication on $\Poly(\g^{*})$ is defined by
$$
(f\ot x)*(g\ot y)=f\bar{x}g\ot y,
$$
where $f,g\in S(\g_{Lie})$, $x,y\in\g$
and $\bar{x}$ is the image of $x$ by the Liezation $\g\to\g_{Lie}$.
\begin{lemma}
The space $(\Poly(\g^{*}),*)$ becomes a perm-algebra.
\end{lemma}
We define an action of $S(\g_{Lie})$ to $\Poly(\g^{*})$.
For any $f\in S(\g_{Lie})$, $g\ot y\in\Poly(\g^{*})$,
\begin{eqnarray}
\label{defact1}\{f,g\ot y\}&:=&\{f,g\}\ot y+g\{f,y\},\\
\label{defact2}\{\bar{x}_{1}\cdot\cdot\cdot\bar{x}_{n},y\}&:=&
\sum_{i}
\bar{x}_{1}\cdot\cdot\cdot
\bar{x}_{i-1}\bar{x}_{i+1}\cdot\cdot\cdot
\bar{x}_{n}\ot[x_{i},y],
\end{eqnarray}
where
$\{f,g\}$ is the Poisson brackets on $S(\g_{Lie})$,
$\bar{x}_{1}\cdot\cdot\cdot\bar{x}_{n}$ is a monomial in $S(\g_{Lie})$
and the bracket $\{f,y\}$ in (\ref{defact1})
is defined by (\ref{defact2})
which is well-defined (cf. (\ref{fact})).
For example,
$\{\bar{x}\bar{y},g\ot z\}=\{\bar{x}\bar{y},g\}\ot z+
g\bar{x}\ot[y,z]+g\bar{y}\ot[x,z]$.
\begin{definition}
For any $f\ot x,g\ot y\in \Poly(\g^{*})$,
\begin{equation}\label{defbra0}
\{f\ot x,g\ot y\}:=\{f\bar{x},g\ot y\}.
\end{equation}
\end{definition}
We remark that the bracket (\ref{defbra0}) is 
not skewsymmetric, even if $\g$ is Lie.
\begin{lemma}
Define a projection $\Poly(\g^{*})\to\lbar{S}(\g_{Lie})$ by
$$
\overline{f\ot x}=f\bar{x},
$$
where $\lbar{S}(\g_{Lie})$ is the Lie-Poisson algebra of non-unital.
The projection preserves the perm-multiplication and the bracket on $\Poly(\g^{*})$,
\begin{eqnarray*}
\overline{(f\ot x)*(g\ot y)}&=&(f\bar{x})(g\bar{y}),\\
\overline{\{f\ot x,g\ot y\}}&=&\{f\bar{x},g\bar{y}\}.
\end{eqnarray*}
\end{lemma}
The main result of this note is as follows.
\begin{theorem}\label{theorema}
The algebra $(\Poly(\g^{*}),*,\{,\})$ is the free dual-prePoisson algebra over $\g$
and $\g$ is a subalgebra of $\Poly(\g^{*})$.
Namely $\Poly(\g^{*})$ satisfies the conditions (A1)-(A4)
in Introduction.
\end{theorem}
\begin{proof}
We show that the bracket (\ref{defbra0}) satisfies the Leibniz identity.
It suffices to prove that the action (\ref{defact1}) is {\em Lie algebraic}.
\begin{lemma}\label{lemmalemma}
For any $f,g\in S(\g_{Lie})$, $h\ot z\in \Poly(\g^{*})$,
$L(f,g,h\ot z)=0$, where
$L$ is the Leibnizator
$L(1,2,3):=\{1,\{2,3\}\}-\{\{1,2\},3\}-\{2,\{1,3\}\}$.
\end{lemma}
\begin{proof}
By the biderivation property, we have
$$
L(f,g,h\ot z)=L(f,g,h)\ot z+hL(f,g,z)=hL(f,g,z).
$$
Here the Jacobi identity $L(f,g,h)=0$ is used.
So it suffices to show that $L(f,g,z)=0$.
When $f=\bar{x}$ and $g=\bar{y}$,
$L(\bar{x},\bar{y},z)=0$ is equivalent to
the Leibniz identity of $\g$.
By the biderivation properties again,
$$
L(f\bar{x},g,z)=fL(\bar{x},g,z)+\bar{x}L(f,g,z).
$$
By the assumption of induction with respect to
the degree of polynomials, we have
$L(\bar{x},g,z)=L(f,g,z)=0$,
which gives $L(f\bar{x},g,z)=0$.
In the same way $L(f,g\bar{x},z)=0$ is shown.
\end{proof}
From this lemma, we obtain the desired identity,
$$
L(f\ot x,g\ot y,h\ot z)
=L(f\bar{x},g\bar{y},h\ot z)=0.
$$
By a direct computation, one can check
the axioms (\ref{dpp1})-(\ref{dpp3}).
Thus, $\Poly(\g^{*})$ satisfies (A1) and (A2) in Introduction.
The Leibniz algebra $\g$ is identified with
$1\ot \g$ as the linear subalgebra of $\Poly(\g^{*})$. Namely (A3) holds.\\
\indent
We show that $\Poly(\g^{*})$ is the free dual-prePoisson algebra over $\g$.
\begin{lemma}
Given a dual-prePoisson algebra $P$ and
given a Loday algebra homomorphism $\phi:\g\to P$,
there exists the unique dual-prePoisson algebra morphism
$\hat{\phi}:\Poly(\g^{*})\to P$, which is defined by
$$
\hat{\phi}:
(\bar{x}_{1}\bar{x}_{2}\cdot\cdot\cdot\bar{x}_{n})\ot y
\mapsto
\phi(x_{1})*\phi(x_{2})*\cdot\cdot\cdot*\phi(x_{n})*\phi(y).
$$
\end{lemma}
\begin{proof}
The mapping $\hat{\phi}$ is well-defined
because $P$ satisfies $(1*2)*3=(2*1)*3$ and
$\{1,2\}_{P}*3=-\{2,1\}_{P}*3$.
Here $\{,\}_{P}$ is the dual-prePoisson bracket on $P$.
It is obvious that $\hat{\phi}$ preserves the perm-product.
We show that $\hat{\phi}$ preserves the dual-prePoisson bracket.
The defining eq. (\ref{defact2}) is preserved:
\begin{eqnarray*}
\hat{\phi}\{\bar{x}_{1}\cdots\bar{x}_{n},y\}&=&
\hat{\phi}\sum_{i}
\bar{x}_{1}\cdots
\bar{x}_{i-1}\bar{x}_{i+1}\cdots
\bar{x}_{n}\ot[x_{i},y]\\
&=&\sum_{i}
\phi(x_{1})*\cdots*\phi(x_{i-1})*\phi(x)_{i+1}*\cdots*
\phi(x)_{n}*\{\phi(x_{i}),\phi(y)\}_{P}\\
&=&\{\phi(x_{1})*\phi(x_{2})*\cdot\cdot\cdot*\phi(x_{n}),\phi(y)\}_{P}.
\end{eqnarray*}
From
$\hat{\phi}(\{\bar{x},\bar{y}\}\ot z)=
\hat{\phi}(\overline{[x,y]}\ot z)=
\phi[x,y]*\phi(z)=
\{\phi(x),\phi(y)\}*\phi(x)$,
for any $f,g\in S(\g_{Lie})$ we have
$$
\hat{\phi}(\{f,g\}\ot z)=\{\hat{\phi}(f),\hat{\phi}(g)\}*\phi(z),
$$
which implies that (\ref{defact1}) is also preserved.
\end{proof}
Thus, we know that $\Poly(\g^{*})$ satisfies (A1)-(A4).
The proof of the theorem is completed.
\end{proof}
\noindent
\textbf{Proof of Proposition \ref{prop0}.}
Since the free Poisson algebra is unital, there exists an injection,
$$
V\to 1\ot\F_{\Lod}(V)\to\F_{\Poiss}(V)\ot\F_{\Lod}(V).
$$
The Liezation of the free Loday algebra
is the free Lie algebra, because the Liezation
is universal.
It is well-known that $\F_{\Poiss}=\F_{\Com}\F_{\Lie}$.
Hence,
$$
\F_{\Poiss}(V)=\F_{\Com}
\Big(\big(\F_{\Lod}V\big)_{Lie}\Big).
$$
Thus $\F_{\Poiss}(V)\ot\F_{\Lod}(V)$
becomes a dual-prePoisson algebra.
The universality is followed from the diagram,
$$
\begin{CD}
V@>>>\F_{\Lod}(V)@>>>\F_{\Poiss}(V)\ot\F_{\Lod}(V) \\
@VVV @VVV @VVV \\
A@= A@= A,
\end{CD}
$$
where $A$ is an arbitrary dual-prePoisson algebra.
The proof of Proposition \ref{prop0} is completed.
\medskip\\
\indent
We notice that
the projection $\Poly(\g^{*})\to\lbar{S}(\g_{Lie})$
coincides with the Poissonization of $\Poly(\g^{*})$,
that is,
$$
\Big(\Poly(\g^{*})\Big)_{Poiss}\cong\lbar{S}(\g_{Lie}).
$$
There exists a natural morphism $(I_{1},I_{0})$ in $\LM$,
$$
\begin{CD}
\g@>I_{1}>>\Poly(\g^{*})\\
@VVV @VVV \\
\g_{Lie}@>I_{0}>>\lbar{S}(\g_{Lie}).
\end{CD}
$$
It is natural to ask whether or not $(I_{1},I_{0})$ is
a universal arrow.
\begin{corollary}
The Poissonization $\Poly(\g^{*})\to\lbar{S}(\g_{Lie})$ is
free over the Liezation,
that is, this Poisson object is the ``Lie-Poisson object"
over the Lie object.
\end{corollary}

%%%%%%%%%%%%%%%%%%%%%%%%%%%%%%%%%%%%
\subsection{Smooth and graded cases}
%%%%%%%%%%%%%%%%%%%%%%%%%%%%%%%%%%%%

We consider a smooth case.
Assume $\dim\g<\infty$ and replace (\ref{defact2}) to
$$
\{f,y\}=\sum_{i}\frac{\pa f}{\pa \bar{e}_{i}}\ot [e_{i},y],
$$
where $f\in C^{\infty}(\g_{Lie})$ and
$\{\bar{e}_{i}\}$ is a linear coordinate on $\g^{*}_{Lie}$
which is identified with a linear basis of $\g$.
Then the definition (\ref{defbra0}) can be extended
to $C^{\infty}(\g^{*}_{Lie})\ot\g$.
\begin{corollary}
[Smooth cases]
We put
$\A^{\infty}(\g^{*}):=C^{\infty}(\g^{*}_{Lie})\ot\g$.
Then $\A^{\infty}(\g^{*})$ is a dual-prePoisson algebra.
\end{corollary}
\begin{proof}
Lemma \ref{lemmalemma} is shown
by a direct computation without induction.
\end{proof}
We consider the super (graded) cases.
Let $V$ be a usual vector space.
The parity change of $V$, which is denoted by $\Pi V$,
is a superspace whose structure sheaf is
the Grassmann algebra over the dual space $V^{*}$,
$$
C^{\infty}(\Pi V):=\bigwedge^{\cdot\ge 0}V^{*}.
$$
If $V=\g^{*}_{Lie}$ the dual space of a Lie algebra, then
$C^{\infty}(\Pi\g^{*}_{Lie})$ becomes a Poisson algebra of degree $(-1,0)$.
Its Poisson bracket is known as Schouten-Nijenhuis (SN) bracket.
In the cases of Loday algebras, the superspace is not
commutative algebra but perm-algebra.
We introduce the Loday algebra version of SN-bracket.
\begin{corollary}
[Noncommutative Schouten-Nijenhuis brackets]
Given a Loday algebra $\g$,
we define a nonstandard superspace $\Pi\g^{*}$ as
$$
\A^{\infty}(\Pi\g^{*}):=C^{\infty}(\Pi\g^{*}_{Lie})\ot(\uparrow\g).
$$
Here $\uparrow\g$ is the linear functions on $\g^{*}$ with the degree $+1$.
Then $\A^{\infty}(\Pi\g)$ becomes a super (graded)
dual-prePoisson algebra with the degree $(-1,0)$.
\end{corollary}
\begin{definition}
We call the dual-prePoisson algebras
$\Poly(\g^{*})$, $\A^{\infty}(\g^{*})$ and $\A^{\infty}(\Pi\g^{*})$
the \textbf{Loday-Poisson algebras}.
\end{definition}
\begin{example}
Let us consider the case of $\g:=sl(2)$.
Since $sl(2)=sl(2)_{Lie}$,
$$
\A^{\infty}(\Pi sl^{*}(2))=\big(\wedge^{\cdot}sl(2)\big)\ot\uparrow sl(2).
$$
In the classical case, $r=X\wedge H$
is a (triangular-)$r$-matrix which satisfies
the integrability condition $\{r,r\}_{SN}=0$,
where $\{,\}_{SN}$ is a classical
Schouten-Nijenhuis bracket on $\bigwedge^{\cdot}\g$.
In the noncommutative case, we obtain
\begin{eqnarray*}
\{X\ot H,X\ot H\}_{NSN}&=&\{X\wedge H,X\ot H\}\\
&=&X\wedge\{H,X\ot H\}-H\wedge\{X,X\ot H\}\\
&=&2H\wedge X\ot X,
\end{eqnarray*}
where $[H,X]=2X$ is used.
\end{example}

%%%%%%%%%%%%%%%%%%%%%%%
\section{Quantization}
%%%%%%%%%%%%%%%%%%%%%%%

In this section we study an algebraic quantization
of Loday-Poisson algebra.
In first we recall the notion of associative dialgebra.
\begin{definition}[\cite{Lod2}]
Let $D$ be a vector space with
two associative multiplications $\vd$ and $\dv$.
When the following three axioms are satisfied,
$D$ is called an associative dialgebra,
or called an associative Loday algebra,
\begin{eqnarray*}
(a\vd b)\vd c&=&(a\dv b)\vd c,\\
(a\vd b)\dv c&=&a\vd(b\dv c),\\
a\dv(b\dv c)&=&a\dv(b\vd c),
\end{eqnarray*}
where $a,b,c\in D$.
\end{definition}
\indent
If $D$ is a dialgebra, then
the commutator
$$
[a,b]_{di}:=a\vd b-b\dv a
$$
is a Loday bracket.
Hence, dialgebras are
associative analogues of Loday algebras.\\
\indent
Given a Loday algebra $\g$,
the universal enveloping {\em dialgebra},
$Ud(\g)$, has the form of $Ud(\g):=U(\g_{Lie})\ot\g$.
The dialgebra multiplications on $Ud(\g)$
are defined by
\begin{eqnarray*}
&&f\ot x \vd g\ot y:=f\bar{x}g\ot y,\\
&&f\ot x \dv 1\ot y:=f\bar{y}\ot x-f\ot [y,x],\\
&&f\ot x \dv (\bar{y}_{1}\cdots\bar{y}_{n-1}\ot y_{n}):=
(\cdots(f\ot x \dv 1\ot y_{1})\dv\cdots)\dv(1\ot y_{n}),
\end{eqnarray*}
where $f,g,\bar{x}_{\cdot}\in U(\g_{Lie})$.
We have
$$
(1\ot x)\vd(1\ot y)-(1\ot y)\dv(1\ot x)=1\ot [x,y].
$$
Since $grU(\g_{Lie})\cong S(\g_{Lie})$ (PBW-theorem),
we obtain $grUd(\g)=\Poly(\g^{*})$.
\begin{remark}
[\cite{LP}]
It is known that
$Ud(\g)\to\lbar{U}(\g_{Lie})$ is
the universal enveloping gebra,
which is an associative gebra (or associative object) in $\LM$,
over the Lie object $\g\to\g_{Lie}$.
\label{remarklp}
\end{remark}
\indent
Let $S_{\star}(\g_{Lie}):=
\big(S(\g_{Lie})[[\hbar]],\star\big)$
be a canonical deformation quantization (Kontsevich \cite{Kont})
of the Lie-Poisson algebra, where
$\dim\g<\infty$ and
$\hbar$ is the formal parameter of deformation.
In \cite{Kont} Section 8.3.1, it was shown that
$S_{\star}(\g_{Lie})$ is isomorphic to
$U_{\hbar}(\g_{Lie})$ as an associative algebra,
where $U_{\hbar}(\g_{Lie})$ is the universal
enveloping algebra over $\g_{Lie}$ with the Lie bracket $\hbar[,]$.
We recall the proof of this theorem.
Assume $\hbar:=1$.
The star-product satisfies
$$
\bar{x}\star\bar{y}=\bar{x}\bar{y}
+\frac{1}{2}\{\bar{x},\bar{y}\},
$$
on the level of generators.
Since $\{\bar{x},\bar{y}\}=[\bar{x},\bar{y}]$,
we have
$\bar{x}\star\bar{y}-\bar{y}\star\bar{x}=[\bar{x},\bar{y}]$,
which gives an algebra homomorphism,
$I:U(\g_{Lie})\to S_{\star}(\g_{Lie})$.
This mapping is clearly surjective
and preserves the top terms of polynomials.
Hence it is bijective.
When $\hbar\neq 1$,
$I$ is the map of $U_{\hbar}(\g_{Lie})$ to $S_{\star}(\g_{Lie})$,
because the degree of monomial is coherent with the one of $\hbar$.\\
\indent
We extend $I$ to the following isomorphism,
$$
I\ot 1: Ud_{\hbar}(\g)\to\Poly_{\star}(\g^{*}),
$$
where $\Poly_{\star}(\g^{*}):=S_{\star}(\g_{Lie})\ot\g$
and where $Ud_{\hbar}(\g)$ is the universal enveloping dialgebra
over $\g$ with the Loday bracket $\hbar[,]$.
Via the isomorphism, one can define the unique
dialgebra structure on $\Poly_{\star}(\g^{*})$.
Let us denote by $\vd_{\star}$ and $\dv_{\star}$
the dialgebra multiplications on $\Poly_{\star}(\g^{*})$.
For example,
\begin{eqnarray*}
g\ot y\dv_{\star}\bar{x}_{1}\ot x_{2}&=&
(I\ot 1)\Big(I^{-1}(g)\ot y\dv \bar{x}_{1}\ot x_{2}\Big)\\
&=&
(I\ot 1)\Big((I^{-1}(g)\bar{x}_{1}\ot y-I^{-1}(g)\ot\hbar[x_{1},y])\dv 1\ot x_{2}\Big)\\
&=&
(I\ot 1)\Big(
I^{-1}(g)\bar{x}_{1}\bar{x}_{2}\ot y
-I^{-1}(g)\bar{x}_{1}\ot\hbar[x_{2},y]
-I^{-1}(g)\bar{x}_{2}\ot\hbar[x_{1},y]\\
&+&I^{-1}(g)\ot\hbar^{2}[x_{2},[x_{1},y]]
\Big)\\
&=&
g\star\bar{x}_{1}\star\bar{x}_{2}\ot y
-\hbar g\star\bar{x}_{1}\ot[x_{2},y]
-\hbar g\star\bar{x}_{2}\ot[x_{1},y]
+\hbar^{2}g\ot[x_{2},[x_{1},y]],
\end{eqnarray*}
where $I(\bar{x})=\bar{x}$ is used.
\begin{proposition}
For any $f\ot x,g\ot y\in\Poly_{\star}(\g^{*})$,
\begin{eqnarray*}
&&\lim_{h\to 0}(f\ot x)\vd_{\star}(g\ot y)=(f\ot x)*(g\ot y),\\
&&\lim_{h\to 0}\frac{1}{\hbar}[f\ot x,g\ot y]_{di}=
\{f\ot x,g\ot y\}.
\end{eqnarray*}
\end{proposition}
One concludes that
$\Poly_{\star}(\g^{*})$ is a deformation quantization of
the Loday-Poisson algebra $\Poly(\g^{*})$.
\begin{proof}
By the definition of the star dialgebra products,
\begin{equation}\label{qprove1}
\bar{x}_{1}\star\cdots\star\bar{x}_{n-1}\ot x_{n}\vd_{\star}g\ot y=
\bar{x}_{1}\star\cdots\star\bar{x}_{n}\star g\ot y
\end{equation}
and
\begin{multline*}
g\ot y\dv_{\star}(\bar{x}_{1}\star\cdots\star\bar{x}_{n-1}\ot x_{n})=\\
g\star\bar{x}_{1}\star\cdots\star\bar{x}_{n}\ot y-
\hbar\sum_{i\ge 1}
g\star\bar{x}_{1}\star\cdots\star\bar{x}^{\vee}_{i}\star\cdots\star\bar{x}_{n}
\ot[x_{i},y]+\cdots=\\
g\star\bar{x}_{1}\star\cdots\star\bar{x}_{n}\ot y-
\hbar\sum_{i\ge 1}
g\bar{x}_{1}\cdots\bar{x}^{\vee}_{i}\cdots\bar{x}_{n}\ot[x_{i},y]+\cdots.
\end{multline*}
Eq. (\ref{qprove1}) yields the first identity of the proposition
because $f$ is a polynomial with respect to the star product.
The commutator is
\begin{multline*}
\lim_{\hbar\to 0}
\frac{1}{\hbar}[\bar{x}_{1}\star\cdots\star\bar{x}_{n-1}\ot x_{n},g\ot y]_{di}=\\
\lim_{\hbar\to 0}\frac{1}{\hbar}
[\bar{x}_{1}\star\cdots\star\bar{x}_{n},g]\ot y+
\sum_{i\ge 1}
g\bar{x}_{1}\cdots\bar{x}^{\vee}_{i}\cdots\bar{x}_{n}\ot[x_{i},y]=\\
=\{\bar{x}_{1}\cdots\bar{x}_{n-1}\ot x_{n},g\ot y\},
\end{multline*}
where
$$
\lim_{\hbar\to 0}\frac{1}{\hbar}[\bar{x}_{1}\star\cdots\star\bar{x}_{n},g]
=\{\bar{x}_{1}\cdots\bar{x}_{n},g\}.
$$
The proof of the proposition is completed.
\end{proof}
\begin{remark}
[recall Remark \ref{remarklp}]
The associative object,
$\Poly_{\star}(\g^{*})\to\lbar{S}_{\star}(\g_{Lie})$,
can seen as a quantization of the Lie-Poisson object;
$$
\begin{CD}
\g@>Poiss_{1}>>\Poly(\g)@>{Quanti_{1}}>>\Poly_{\star}(\g^{*}) \\
@VVV @VVV @VVV \\
\g_{Lie}@>{Poiss_{0}}>>\lbar{S}(\g_{Lie})@>{Quanti_{0}}>>\lbar{S}_{\star}(\g_{Lie}).
\end{CD}
$$
Here $Poiss_{\cdot}$ is the Lie-Poisson functor
and $Quanti_{\cdot}$ is the quantization functor.
\end{remark}
\noindent
\textbf{Final Remark}. We studied six types of algebras,
i.e., three classical algebras; Lie, Poisson, Associative
and their noncommutative analogues; Loday, Loday-Poisson, di-associative.
We recall their operads
(see \cite{GK1,GK2}, \cite{Lod2}, \cite{Val} and \cite{MSS}
for operads).
\begin{definition}
[\cite{U}]
Let $\P$ be a binary quadratic operad
and let $\Perm$ be the operad of perm-algebras.
We call the functor
$$
\P\mapsto\Perm\ot\P
$$
a derived bracket construction, on the level of operad. 
\end{definition}
It is known that the operads of the six algebras are related
via the derived bracket constriction
\begin{eqnarray*}
\Lod &=& \Perm\ot\Lie,\\
\Lod\Poiss &=& \Perm\ot\Poiss,\\
\Dias &=& \Perm\ot\Ass,
\end{eqnarray*}
where $\Lie$, $\Poiss$ and $\Ass$ are respectively
operads of Lie algebras, Poisson algebras and associative algebras
and $\Lod$, $\Lod\Poiss$ and $\Dias$
are respectively operads of Loday algebras,
Loday-Poisson algebras (or dual-prePoisson algebras)
and associative dialgebras.
Thus, we obtain an operad theoretical description for
the results of the previous sections.
\begin{center}
\begin{tabular}{|l|l|l|l|} 
\hline
Type& Symmetry & Classical & Quantum \\ \hline
$\Perm\ot\P$& $\Lod$ & $\Lod\Poiss$ & $\Dias$ \\ \hline
$\P$ & $\Lie$ & $\Poiss$ & $\Ass$ \\
\hline
\end{tabular}
\end{center}

\begin{verbatim}
Kyousuke UCHINO
Shinjyuku Tokyo Japan
email:K_Uchino@oct.rikadai.jp
\end{verbatim}

\begin{thebibliography}{}
\bibitem{A}
M. Aguiar.
Pre-Poisson algebras.
Lett. Math. Phys. 54 (2000), no. 4, 263--277.
\bibitem{Chap}
F. Chapoton.
Un endofoncteur de la categorie des operades.
Lecture Notes in Mathematics, 1763.
Springer-Verlag, Berlin (2001), 105--110.
\bibitem{Chap2}
F. Chapoton and M. Livernet.
Pre-Lie algebras and the rooted trees operad.
Internat. Math. Res. Notices (2001), no. 8, 395--408. 
\bibitem{GK1}
V. Ginzburg and M. Kapranov.
Koszul duality for operads.
Duke Math. J. 76 (1994), no. 1, 203--272.
\bibitem{GK2}
V. Ginzburg and M. Kapranov.
Erratum to: ``Koszul duality for operads''
Duke Math. J. 80 (1995), no. 1, 293.
\bibitem{ILMP}
R. Ibanez, M. de Leon, J. C. Marrero and E. Padron.
Leibniz algebroid associated with a Nambu-Poisson structure.
J. Phys. A: Math. Gen. 32 (1999), 8129.
%\bibitem{Kos1}
%Y. Kosmann-Schwarzbach.
%From Poisson algebras to Gerstenhaber algebras.
%Ann. Inst. Fourier (Grenoble) 46 (1996), no. 5, 1243--1274.
%\bibitem{Kos2}
%Y. Kosmann-Schwarzbach.
%Derived brackets.
%Lett. Math. Phys. 69 (2004), 61--87.
\bibitem{Kont}
M. Kontsevich.
Deformation quantization of Poisson manifolds.
Lett. Math. Phys. 66 (2003), 157--216.
\bibitem{Lod1}
J-L. Loday and T. Pirashvili.
Universal enveloping algebras of Leibniz algebras
and (co)homology.
Math. Ann. 296 (1993), no. 1, 139--158.
\bibitem{LP}
J-L. Loday and T. Pirashvili.
The tensor category of linear maps and Leibniz algebras.
Georg. Math. J. vol 5 (1998), no. 3, 263--276.
\bibitem{Lod2}
J-L. Loday.
Dialgebras.
Lecture Notes in Mathematics, 1763.
Springer-Verlag, Berlin, (2001), 7--66.
\bibitem{MSS}
M. Markl, S. Shnider and J. Stasheff.
Operads in algebra, topology and physics.
Mathematical Surveys and Monographs, 96.
American Mathematical Society, Providence,
RI, (2002). x+349 pp.
\bibitem{Val}
B. Vallette.
Manin products, Koszul duality, Loday algebras
and Deligne conjecture.
J. Reine Angew. Math. 620 (2008), 105--164.
\bibitem{U}
K. Uchino.
Derived bracket construction and Manin products.
Lett. Math. Phys. 90 (2010), 37--53.
\end{thebibliography}
\end{document}